\documentclass[oneside,english]{amsart}
\usepackage[T1]{fontenc}
\usepackage[latin9]{inputenc}
\usepackage{geometry}
\usepackage{mathrsfs}
\usepackage{mathtools}
\usepackage{amstext}
\usepackage{amsthm}
\usepackage{amssymb}

\makeatletter

\newcommand{\noun}[1]{\textsc{#1}}
\providecommand{\tabularnewline}{\\}

\numberwithin{equation}{section}
\numberwithin{figure}{section}
\theoremstyle{plain}
\newtheorem{thm}{\protect\theoremname}
\theoremstyle{plain}
\newtheorem{lem}[thm]{\protect\lemmaname}
\theoremstyle{plain}
\newtheorem{cor}[thm]{\protect\corollaryname}

\usepackage{bbold}
\usepackage[all]{xy}

\makeatother

\usepackage{babel}
\providecommand{\corollaryname}{Corollary}
\providecommand{\lemmaname}{Lemma}
\providecommand{\theoremname}{Theorem}

\begin{document}
\title{\noun{All Hurwitz Algebras from 3D Geometric Algebras}}
\author{Daniele Corradetti$^{*}$, Richard Clawson$^{**}$, Klee Irwin$^{**}$}
\begin{abstract}
Hurwitz algebras are unital composition algebras widely known in algebra
and mathematical physics for their useful applications. In this paper,
inspired by works of Lesenby and Hitzer, we show how to embed all
seven Hurwitz algebras (division and split) in 3D geometric algebras,
i.e. $\mathcal{G}\left(p,q\right)$ with $p+q=3$. This is achieved
studying the even subalgebra, which is always of quaternionic
or split-quaternionic type, and the algebra generated by the pseudoscalar
which is always of complex or split-complex type. Reversion, inversion
and Clifford conjugation correspond to biquaternionic, complex and
quaternionic conjugation respectively. Octonionic algebras, i.e.,
octonions and split-octonions, are obtained in two ways introducing two different
products, replicating a variation of the Cayley-Dickson process. Octonionic
conjugation is achieved in the geometric algebra formalism through
the involution defined by the full grade inversion.
\end{abstract}

\maketitle

\section{Introduction }

Hurwitz algebras are composition algebras known for their wide applications
in algebra and mathematical physics that range over projective geometry,
Lie theory, ring theory, etc. (see \cite{CCMA-Magic,corr Notes Octo,Corr RealF}
and references therein). They are a distinct class of algebras that
possess a unit element and a non-degenerate norm $n$ satisfying the
multiplicative property, i.e. $n\left(xy\right)=n\left(x\right)n\left(y\right)$.
These algebras are totally classified due to Hurwitz's theorem, which
states that there are only four division algebras over the real numbers:
the real numbers $\mathbb{R}$, the complex numbers $\mathbb{C}$,
the quaternions $\mathbb{H}$, and the octonions $\mathbb{O}$. Additionally,
there are their respective split counterparts, i.e., the split-complex
numbers $\mathbb{C}_{s}$, the split-quaternions $\mathbb{H}_{s}$, and
the split-octonions $\mathbb{O}_{s}$ forming a total of seven Hurwitz algebras.
The key to understanding these algebras lies in their properties.
For instance, $\mathbb{R}$ is commutative, associative, and totally
ordered, $\mathbb{C}$ is commutative and associative, $\mathbb{H}$
is associative, and $\mathbb{O}$ is non-associative but alternative.
The split counterparts share the same properties, except they are
not division algebras and contain zero divisors. The non-associativity
of the octonions and the split octonions were until recentely one
of the main obstacles to the realisation of such algebras through
Clifford algebras, whose product is associative. Indeed, previously
known embeddings in Clifford algebra for octonions were in high dimensions
such as $Cl\left(0,7\right)$, which is of dimension 128, or $Cl\left(8,0\right)$,
which is of dimension 256 (see \cite[Ch. 7.3 and 23]{Lo01}). Lasenby
\cite{La22} and Hitzer \cite{Hi22} used a different approach, constructing a novel product from Clifford's geometric product over independent components and achieving embeddings of the
octonions even in 3D geometric algebras, i.e. $\mathcal{G}\left(p,q\right)$
with $p+q=3$.

This paper completes Lasenby's and Hitzer's work, showing for the first
time how to embed all seven Hurwitz algebras in 3D geometric algebras
(division and split versions) following the same natural procedure
and not limited to the bilinear product but extending the treatment
to the conjugation and thus to the norm of those algebras (indeed
a composition algebra is defined by a bilinear product and a quadratic
form under which the algebra is compositional). 

The paper is structured as follows. In section \ref{sec:Hurwitz-algebras}
we review Hurwitz algebras, along with their basic properties and
a simple realisation with a formalism that closely mirrors that of
geometric algebra. We have decided in this section to allocate space
for presenting biquaternionic algebras, i.e. $\mathbb{C\otimes\mathbb{H}}$,
$\mathbb{C}_{s}\otimes\mathbb{H}$, $\mathbb{C\otimes}\mathbb{H}_{s}$
and $\mathbb{C}_{s}\mathbb{\otimes}\mathbb{H}_{s}$, due to their
significance in the subsequent sections. In section \ref{sec:Geometric-algebras},
we revisit the foundations of geometric algebra and their involutions,
such as reversion, inversion and Clifford conjugation, setting the
stage for the subsequent sections. Main results are in section \ref{sec:Hurwitz-algebras-from}.
Firstly, we show the isomorphism between the 3D geometric algebras,
i.e., $\mathcal{G}\left(p,q\right)$ with $p+q=3$, with biquaternionic
algebras (see Table \ref{tab:Correspondence-between-Geometric}).
This is achieved considering the even subalgebra $\mathcal{G}\left(p,q\right)_+,$, which is always of quaternionic or split-quaternionic
type, and the algebra generated by the pseudoscalar $\textit{Ps}\left(\mathcal{G}\right)$
which is always of complex or split-complex type. Then, since the
pseudoscalar subalgebra $\textit{Ps}\left(\mathcal{G}\right)$ commutes with
all the elements of $\mathcal{G}\left(p,q\right)$ one has 
\begin{equation}
\mathcal{G}\left(p,q\right)\cong \textit{Ps}\left(\mathcal{G}\right)\otimes\mathcal{G}\left(p,q\right)_+,
\end{equation}
which is shown in section \ref{subsec:Complex-and-quaternionic}.
In section \ref{subsec:Octonionic-algebras}, we obtain octonions
$\mathbb{O}$ and split-octonions $\mathbb{O}_{s}$ (depending on the signature $(p,q)$) in two ways, first using the product (here denoted $\bullet$) introduced by Lasenby \cite{La22} and Hitzer \cite{Hi22},
replicating a variation of the Cayley-Dickson process for sublgebras
of $\mathcal{G}\left(p,q\right)$, and then using a variation of this product $\bullet_{-}$. To prove the octonionic nature
of the resulting algebras we proceed defining an octonionic norm through
the full grade inversion and then proving that the norm has the compositional
property under the product $\bullet$. Hurwitz's theorem then assures
that, since $\left(\mathcal{G}\left(p,q\right),\bullet,N\right)$
and $\left(\mathcal{G}\left(p,q\right),\bullet_{-},N\right)$ are
both unital and eight-dimensional, then they must be isomorphic to
either the octonions $\mathbb{O}$ or the split-octonions $\mathbb{O}_{s}$.
A direct inspection of the explicit formula for the norm allows the
complete classification given in Table \ref{tab:Correspondence-between-Geometric}.

\section{\label{sec:Hurwitz-algebras}Hurwitz algebras}

An\emph{ algebra}, denoted by $A$, is a vector space over a field
$\mathbb{F}$ (henceforth assumed to be the reals $\mathbb{R}$)
equipped with a bilinear multiplication. The specific properties of
the multiplication operation in an algebra lead to various classifications.
Specifically, an algebra $A$ is said to be\emph{ commutative} if
$x\cdot y=y\cdot x$ for every $x,y\in A$; \emph{associative}
if it satisfies $x\cdot\left(y\cdot z\right)=\left(x\cdot y\right)\cdot z$;
\emph{alternative} if $x\cdot\left(y\cdot y\right)=\left(x\cdot y\right)\cdot y$;
and finally, \emph{flexible} if $x\cdot\left(y\cdot x\right)=\left(x\cdot y\right)\cdot x$. 

Since $A$ must be a group with respect to addition, every algebra
has a zero element $0\in A$. Furthermore, if the algebra does not
have zero divisors, it is referred to as a \emph{division} algebra,
i.e. an algebra for which $x\cdot y=0$ implies $x=0$ or $y=0$.
While the zero element is a universal feature in any algebra, the
algebra is termed \emph{unital} if there exists an element $1\in A$
such that $1\cdot x=x\cdot1=x$ for all $x\in A$. 

Consider an algebra $A$. Then a quadratic form $n$ on $A$ over
the field $\mathbb{F}$ is called a \emph{norm} and its polarization
is given by 
\begin{equation}
\left\langle x,y\right\rangle =n\left(x+y\right)-n\left(x\right)-n\left(y\right),\label{eq:polarNorm}
\end{equation}
so that the norm can be explicitly given as
\begin{equation}
n\left(x\right)=\frac{1}{2}\left\langle x,x\right\rangle ,\label{eq:N(x)=1/2(x,x)}
\end{equation}
for every $x\in A$. An algebra $A$ with a non-degenerate norm $n$
that satisfies the following multiplicative property, i.e.

\begin{align}
n\left(x\cdot y\right) & =n\left(x\right)n\left(y\right),\label{eq:comp(Def)}
\end{align}
for every $x,y\in A$, is called a \emph{composition} algebra and
is denoted with the triple $\left(A,\cdot,n\right)$, or simply
$A$ if there is cause for ambiguity. 

Composition algebras that possess a unit element are called \emph{Hurwitz
algebras} and, from now on, we will always denote the bilinear
product over Hurwitz algebras with the symbol $\bullet$. The interplay
between the multiplicative property of the norm in (\ref{eq:comp(Def)})
and the existence of a unit element, is full of interesting implications.
Indeed, every Hurwitz algebra is endowed with an order-two antiautomorphism
called \emph{conjugation}, defined by 
\begin{equation}
\overline{x}=\left\langle x,1\right\rangle 1-x.\label{eq:conjugation}
\end{equation}
The linearization of the norm, when paired with the composition,
results in the notable relation $\left\langle x\bullet y,z\right\rangle =\left\langle y,\overline{x}\bullet z\right\rangle ,$
which implies that $\overline{x\bullet y}=\overline{y}\bullet\overline{x}$
and 
\begin{equation}
x\bullet\overline{x}=n\left(x\right)1.\label{eq:ConjugNorm}
\end{equation}
Moreover, from the existence of a unit element in a composition algebra
we have that elements with unit norm form a goup and, even more strikingly,
that the whole algebra must be alternative (for a proof see \cite[Prop. 2.2]{ElDuque Comp}). 

A major theorem by Hurwitz proves that the only unital composition
algebras over the reals are $\mathbb{R},\mathbb{C},\mathbb{H}$ and
$\mathbb{O}$ accompanied by their split counterparts $\mathbb{C}_{s},\mathbb{H}_{s},\mathbb{O}_{s}$
(see \cite[Cor. 2.12]{Hurwitz98,ElDuque Comp}). Consequently, there
are only seven Hurwitz algebras, each having real dimensions of 1,
2, 4, or 8. Out of these, four are also division algebras, i.e. $\mathbb{R},\mathbb{C},\mathbb{H}$
and $\mathbb{O}$, while three are split algebras and thus have zero
divisors, i.e. $\mathbb{C}_{s},\mathbb{H}_{s},\mathbb{O}_{s}$. The
properties of such algebras are quite different from one another. More
specifically, $\mathbb{R}$ is also totally ordered, commutative and
associative; $\mathbb{C}$ is just commutative and associative; $\mathbb{H}$
is only associative and, finally, $\mathbb{O}$ is only alternative.
\begin{table}
\begin{centering}
\begin{tabular}{|c|c|c|c|c|c|}
\hline 
\textbf{Hurwitz} & \textbf{O.} & \textbf{C.} & \textbf{A.} & \textbf{Alt.} & \textbf{F.}\tabularnewline
\hline 
\hline 
$\mathbb{R}$ & Yes & Yes & Yes & Yes & Yes\tabularnewline
\hline 
$\mathbb{C}$, $\mathbb{C}_{s}$ & No & Yes & Yes & Yes & Yes\tabularnewline
\hline 
$\mathbb{H}$,$\mathbb{H}_{s}$ & No & No & Yes & Yes & Yes\tabularnewline
\hline 
$\mathbb{O}$,$\mathbb{O}_{s}$ & No & No & No & Yes & Yes\tabularnewline
\hline 
\end{tabular}
\par\end{centering}
\centering{}{\small{}\bigskip{}
}\caption{\emph{\label{tab:Hurwitz-para-Hurwitz}On the left,} we have summarized
the algebraic properties, i.e. totally ordered (O), commutative (C),
associative (A), alternative (Alt), flexible (F), of all Hurwitz algebras,
namely $\mathbb{R},\mathbb{C},\mathbb{H}$ and $\mathbb{O}$ along
with their split counterparts $\mathbb{C}_{s},\mathbb{H}_{s},\mathbb{O}_{s}$. }
\end{table}
 As shown by Table \ref{tab:Hurwitz-para-Hurwitz} all properties
of $\mathbb{R},\mathbb{C},\mathbb{H}$ and $\mathbb{O}$ are valid
also for the split companions, the only difference being that the latter
are not division algebras and do have zero divisors. Generalizations
of the Hurwitz Theorem can be done over arbitrary fields (see \cite[p. 32]{ZSSS})
but for our purposes this will not be needed.

\subsection{\label{subsec:Complex-and-split-complex}Complex and split-complex }

Both the complex numbers $\mathbb{C}$ and the split-complex numbers $\mathbb{C}_{s}$
are bidimensional algebras over the reals. Given a basis $\left\{ 1,e_{1}\right\} $,
where $1$ indicates the unit element, the bilinear product of the
complex numbers is defined by the non-trivial relation
\begin{equation}
\begin{array}{ccc}
e_{1}\bullet e_{1} & = & -1,\end{array}\label{eq:complex1}
\end{equation}
 while for the split-complex numbers one requires

\begin{equation}
\begin{array}{ccc}
e_{1}\bullet e_{1} & = & 1.\end{array}\label{eq:splitcomplex1}
\end{equation}
From the relations above it is straightforward to see that both algebras
are unital, commutative and associative. Nevertheless, while the complex
numbers are a division algebra, the split-complex numbers are not.
Indeed, if $e_{1}\bullet e_{1}=1,$ then $\left(e_{1}+1\right)\left(e_{1}-1\right)=0$.
In the longstanding notational tradition dating back to Euler, one
usually sets $e_{1}=\text{\textbf{i}}$ and designates $\left\{ 1,\text{\textbf{i}}\right\} $
as the basis for both the complex numbers $\mathbb{C}$ and the split-complex numbers
$\mathbb{C}_{s}$. As for the conjugation, in both cases one
has the involution defined by
\begin{equation}
\text{\textbf{i}}\longrightarrow\overline{\text{\textbf{i}}}=-\text{\textbf{i}}.
\end{equation}
Then, from $\overline{x}\bullet x=n\left(x\right)1$, one obtains
the norm that turns $\mathbb{C}$ and $\mathbb{C}_{s}$ into composition algebras.

\subsection{\label{subsec:Quaternions-and-split-quaternion}Quaternions and split-quaternions }

Quaternions $\mathbb{H}$ and split-quaternions $\mathbb{H}_{s}$
are the only four-dimensional Hurwitz algebras. Given the basis $\left\{ 1,e_{1},e_{2},e_{3}\right\} $,
with $1$ as the unit element, one obtains the quaternionic product
$\mathbb{H}$ by defining a bilinear product with 
\begin{equation}
\begin{array}{ccc}
e_{i}\bullet e_{i} & = & -1,\\
e_{i}\bullet e_{j} & = & -e_{j}\bullet e_{i},\\
e_{i}\bullet e_{i+1} & = & e_{i+2},
\end{array}\label{eq:quaternionicProduct}
\end{equation}
where $i\in\left\{ 1,2,3\right\} $ and operations on indices are
modulo $3$. Similarly, the split-quaternionic product $\mathbb{H}_{s}$
is given by 
\begin{equation}
\begin{array}{cccc}
e_{1}\bullet e_{1} & = & -e_{2}\bullet e_{2} & =-e_{3}\bullet e_{3}=-1,\\
e_{i}\bullet e_{j} & = & -e_{j}\bullet e_{i},\\
e_{i+1}\bullet e_{i} & = & \left(-1\right)^{i}e_{i+2},
\end{array}\label{eq:splitquaternionicProduct-1}
\end{equation}
where $i\in\left\{ 1,2,3\right\} $ and operations on indices are
modulo $3$. 

Straightforward calculations show that the resulting algebras are
both associative (thus alternative and flexible). Nevertheless, the
definition of the quaternionic and split-quaternionic product is evidently
non-commutative. As in the case of complex numbers, we will often
adopt the longstanding convention that sets $e_{1}=\text{\textbf{i}}$,
$e_{2}=\text{\textbf{j}}$ and $e_{3}=\text{\textbf{k}}$ and for
which a basis of both the algebras is given by $\left\{ 1,\text{\textbf{i}},\text{\textbf{j}},\text{\textbf{k}}\right\} $.

As for the conjugation, in both cases one has the involution defined
by
\begin{equation}
e_{i}\longrightarrow\widetilde{e_{i}}=-e_{i}.
\end{equation}
Then, from $x\bullet\widetilde{x}=n\left(x\right)1$, one obtains
the norm that turns both quaternions $\mathbb{H}$
and split-quaternions $\mathbb{H}_{s}$ into composition algebras.

\subsection{\label{subsec:Octonions-and-split-octonions}Octonions and split-octonions}

Octonions $\mathbb{O}$ and split-octonions $\mathbb{O}_{s}$ are
the only eight-dimensional Hurwitz algebras. Given the basis $\left\{ 1,e_{1},e_{2},e_{3},e_{4},e_{5},e_{6},e_{7}\right\} $,
with $1$ as the unit element, one obtains the octonionic product
$\mathbb{O}$ by defining a bilinear product with 
\begin{equation}
\begin{array}{ccc}
e_{i}\bullet e_{i} & = & -1,\\
e_{i}\bullet e_{j} & = & -e_{j}\bullet e_{i},\\
e_{i}\bullet e_{i+1} & = & e_{i+3},
\end{array}\label{eq:octonionicProduct}
\end{equation}
where $i\in\left\{ 1,...,7\right\} $ and operations on indices are
modulo $7$. To obtain the split-octonionic product one has to set
\begin{equation}
\begin{array}{cccc}
e_{i}\bullet e_{i} & = & -1, & \text{ }i\in\left\{ 1,2,3\right\} ,\\
e_{i}\bullet e_{i} & = & 1,\text{ } & i\notin\left\{ 1,2,3\right\} ,\\
e_{i}\bullet e_{j} & = & -e_{j}\bullet e_{i},\\
e_{i}\bullet e_{j} & = & \varepsilon_{ijk}e_{k},
\end{array}\label{eq:SplitOctonionic}
\end{equation}
 where $\varepsilon_{ijk}$ is the is the Levi-Civita symbol with
value $1$ when $\left(ijk\right)\in\left\{ \left(123\right),\left(154\right),\left(176\right),\left(264\right),\left(257\right),\left(374\right),\left(365\right)\right\} $.

As in the previous cases, we will often adopt the notational convention
for which $e_{1}=\text{\textbf{i}}$, $e_{2}=\text{\textbf{j}}$ $e_{3}=\text{\textbf{k}}$,
$e_{4}=\text{\textbf{l}}$. According to this notation a basis for
both the algebras is given by $\left\{ 1,\text{\textbf{i}},\text{\textbf{j}},\text{\textbf{k}},\text{\textbf{l}},\text{\textbf{li}},\text{\textbf{lj}},\text{\textbf{lk}},\right\} $.
As for the conjugation, in both cases one has the involution defined
by
\begin{equation}
e_{i}\longrightarrow e_{i}^{*}=-e_{i},
\end{equation}
with $i\in\left\{ 1,...,7\right\} $. Then, from $x\bullet x^{*}=n\left(x\right)1$,
one obtains the norm that turns both Octonions
$\mathbb{O}$ and split-octonions $\mathbb{O}_{s}$ into composition algebras.

\subsection{\label{subsec:Biquaternionic-algebras}Biquaternionic algebras}

Beside Hurwitz algebras, also biquaternionic algebras are important.
In fact, biquaternionic algebras are eight-dimensional algebras obtained
from tensor products of two Hurwitz algebras $C\otimes H$, one of
which is a complex algebra $C\in\left\{ \mathbb{C},\mathbb{C}_{s}\right\} $
while the other is a quaternionic algebra $H\in\left\{ \mathbb{H},\mathbb{H}_{s}\right\} $. 

In a biquaternionic algebra, every element $x\in C\otimes H$ is of
the form $x=z\otimes q$ where $z\in C$ and $q\in H$. Then the bilinear
product over $C\otimes H$ is defined by
\begin{align}
x\bullet y & =\left(z_{1}z_{2}\right)\otimes\left(q_{1}q_{2}\right),\label{eq:tensorprod}
\end{align}
where in this case, for simplicity and to avoid ambiguity, we used
juxtaposition to indicate the bilinear product over $C$ and over
$H$ respectively. In all biquaternionic algebras a canonical basis
is given by

\begin{equation}
\mathscr{Q}=\left\{ 1,\iota,\text{\textbf{i}},\text{\textbf{j}},\text{\textbf{k}},\iota\text{\textbf{i}},\iota\text{\textbf{j}},\iota\text{\textbf{k}}\right\} ,
\end{equation}
where $\iota\text{\textbf{i}},\iota\text{\textbf{j}},\iota\text{\textbf{k}}$
is just a short notation for $\iota\otimes\text{\textbf{i}},\iota\otimes\text{\textbf{j}},\iota\otimes\text{\textbf{k}}$.
The element $\iota$ commutes with all elements of the basis and $\iota^{2}=\pm1$
depending on $C$ being $\mathbb{C}$ or $\mathbb{C}_{s}$, $\text{\textbf{i}}^{2}=-1$
while $\text{\textbf{j}}^{2}=\text{\textbf{k}}^{2}=\pm1$ depending
on $H$ being $\mathbb{H}$ or $\mathbb{H}_{s}$. In this basis every
element has a canonical decomposition
\begin{equation}
x=z_{0}+z_{1}\text{\textbf{i}}+z_{2}\text{\textbf{j}}+z_{3}\text{\textbf{k}},
\end{equation}
where $z_{0},z_{1},z_{2}$ and $z_{3}$ are (split-)complex numbers.

Finally, note that since every element is of the form $x=z\otimes q$,
then every biquaternionic algebra inherits three different conjugations
\begin{equation}
\widetilde{x}=z\otimes\widetilde{q},\,\,\overline{x}=\overline{z}\otimes q,\,\,x^{\dagger}=\overline{z}\otimes\widetilde{q},\label{eq:biquaternionicConj}
\end{equation}
where $z\longrightarrow\overline{z}$ and $q\longrightarrow\widetilde{q}$
indicate the complex and the quaternionic conjugation respectively.
In the future sections we will make use of all three of these conjugation
and make contact with the corresponding involutions from the geometric
algebra.

\section{\label{sec:Geometric-algebras}Geometric algebras}

For our purposes, geometric algebras are just a special case of Clifford
algebras over the field of reals $\mathbb{R}$. A Clifford algebra
$Cl\left(V,Q\right)$ is an algebra, i.e. a vector space $V$ over
a field $\mathbb{F}$ endowed with a bilinear product, that is unital
and associative, i.e. there exists an element $1\in Cl\left(V,Q\right)$
such that $1x=x1=x$ and is such that $\left(xy\right)z=x\left(yz\right)$,
and equipped with a \emph{quadratic form} $Q$ from the vector space
$V$ to the field $\mathbb{F}$ such that 
\begin{equation}
v^{2}=Q\left(v\right)1,\label{eq:ideal}
\end{equation}
for every $v\in V$. More specifically, the Clifford algebra $Cl\left(V,Q\right)$
is isomorphic to the tensor algebra $T\left(V\right)$ quotiented
by the ideal $\left\langle v\otimes v-Q\left(v\right)1\right\rangle $.
Polarising the quadratic form one obtains 
\begin{equation}
\left\langle v,w\right\rangle =\frac{1}{2}\left(Q\left(v+w\right)-Q\left(v\right)-Q\left(w\right)\right),\label{eq:polarisation}
\end{equation}
and thus we have the following foundational formula
\begin{equation}
vw+wv=2\left\langle v,w\right\rangle .\label{eq:fundamental}
\end{equation}

Setting a basis for the vector space $V$ will then give us a natural
basis for the Clifford algebra $Cl\left(V,\mathbb{F}\right)$. Let
$\left\{ e_{1},...,e_{n}\right\} $ be an orthogonal basis for $V$
with respect to the quadratic form $Q$ , i.e., $\left\langle e_{i},e_{j}\right\rangle =Q\left(e_{i}\right)\delta_{ij}$
. The previous formula (\ref{eq:fundamental}) implies that 
\begin{align}
e_{i}e_{j} & =-e_{j}e_{i},\text{ when }i\neq j\label{eq:rule1}\\
e_{i}^{2} & =Q\left(e_{i}\right).\label{eq:rule2}
\end{align}
 We then have that the set of elements 
\begin{equation}
\mathscr{B}=\left\{ e_{i_{1}}...e_{i_{k}}:1\leq i_{1}\leq\ldots\leq i_{k}\leq n\text{ and }0\leq k\leq n\right\} ,
\end{equation}
is a basis for $Cl\left(V,Q\right)$ and applying (\ref{eq:rule1})
and (\ref{eq:rule2}) one easily finds that the dimension of $Cl\left(V,\mathbb{F}\right)$,
i.e. the cardinality of linearly independent elements of $\mathscr{B}$,
is $2^{n}$.

For our purposes we will call \emph{geometric algebra} any Clifford
algebra over real vector spaces. Since quadratic norms over real vector
spaces are fully classified by their signature $\left(p,q\right)$
we will designate with $\mathcal{G}\left(p,q\right)$ the Clifford
algebra $Cl\left(V,Q\right)$ where the quadratic form has signature
$\left(p,q\right)$. Moreover, henceforth, following an historical
tradition, we will call the Clifford product the \emph{geometric
product} and denote it with the juxtaposition of the two elements,
i.e., $xy$ for every $x,y\in\mathcal{G}\left(p,q\right)$; we define the
\emph{inner product} as the polarisation $\left\langle \cdot,\cdot\right\rangle $
of the quadratic form $Q$ and denote it with the dot product, i.e.
$\left\langle x,y\right\rangle \coloneqq x\cdot y\in\mathbb{R}$;
finally we will call the antisymmetric part of the geometric product the \emph{exterior product} of $x,y\in\mathcal{G}\left(p,q\right)$
and denote it with the wedge
product $x\land y$. Summarising, for every $x,y\in\mathcal{G}\left(p,q\right)$,
we have 
\begin{align}
x\cdot y & \coloneqq\frac{1}{2}\left(xy+yx\right),\label{eq:symmetric}\\
x\land y & \coloneqq\frac{1}{2}\left(xy-yx\right),\label{eq:antisymmetric}
\end{align}
and therefore 
\begin{align}
xy & =x\cdot y+x\land y,\label{eq:productxy}\\
yx & =x\cdot y-x\land y.\label{eq:productyx}
\end{align}

As Clifford algebras are graded algebra, we introduce the grade selection
bracket $\left\langle x\right\rangle _{k}$ to indicate the grade
$k$ part of a multivector $x\in\mathcal{G}\left(p,q\right)$. We
also indicate grade parity with a subscript, $x_{+}\in\mathcal{G}\left(p,q\right)_{+}$
and $x_{-}\in\mathcal{G}\left(p,q\right)_{-}$ depending if the parity
is even or odd respectively.

Four involutions are naturally defined over over geometric algebras:
\emph{reversion}, \emph{inversion}, \emph{Clifford conjugation} and
\emph{full grade inversion}. 

We define \emph{reversion} as the anti-involution $x\longrightarrow x^{\dagger}$
that reverses the order of vector factors in any product, i.e., the
involution defined axiomatically by 
\begin{equation}
\left\langle x\right\rangle _{1}^{\dagger}=\left\langle x\right\rangle _{1},\left(xy\right)^{\dagger}=y^{\dagger}x^{\dagger}.
\end{equation}
To show how this involution acts on a basis, let $x$ be a rank-$k$
multivector given by the geometric product $x=e_{1}e_{2}...e_{k}$,
with $e_{1},...,e_{k}\in\mathbb{R}^{n}$. Then $x^{\dagger}$ is the
product of the same vectors in the reversed order, i.e.,
\begin{align}
x^{\dagger} & =e_{k}...e_{2}e_{1}.
\end{align}
Another involution is called \emph{inversion, }here denoted by $x\longrightarrow\overline{x}$,
which changes the sign of each vector, and hence of each odd multivector,
i.e.
\begin{equation}
\overline{x}=x_{+}-x_{-}.
\end{equation}
\emph{Clifford conjugation}, denoted by $x\longrightarrow\widetilde{x}$,
combines reversion with inversion so that
\begin{equation}
\widetilde{x}=\overline{x}^{\dagger}.
\end{equation}
Finally, the canonical fourth involution is called the \emph{full
grade inversion} $x\longrightarrow x^{*}$ which changes sign of all
non-zero grades, i.e. 
\begin{equation}
x^{*}=2\left\langle x\right\rangle _{0}-x=x_{+}^{\dagger}-x_{-}.\label{eq:fullGradeInversion}
\end{equation}
As we will see in the next section, all these involutions will be
of major importance for this study. 

\section{\label{sec:Hurwitz-algebras-from}Hurwitz algebras from 3D Geometric
algebras }

We now concretely analyze 3D geometric algebras, i.e., the geometric
algebras over the three-dimensional space $\mathbb{R}^{3}$. Let $\mathcal{G}\left(p,q\right)$
be the the Clifford algebra $Cl_{p,q}\left(\mathbb{R}^{3}\right)$,
i.e., where the quadratic form $Q$ has signature $\left(p,q\right)$.
Since the vector space $\mathbb{R}^{3}$ is of dimension $3$ over
the reals, then $p+q=3$ and the dimension of the resulting geometric
algebra is 8-dimensional, i.e., $\text{dim}\mathcal{G}\left(p,q\right)=2^{3}=8$. 

Let $\left\{ e_{1},e_{2},e_{3}\right\} $ be an orthonormal basis
for $\mathbb{R}^{3}$, then 
\begin{equation}
\mathscr{B}=\left\{ 1,e_{1},e_{2},e_{3},e_{1}e_{2},e_{2}e_{3},e_{1}e_{3},e_{1}e_{2}e_{3}\right\} ,\label{eq:Basis3D-GA}
\end{equation}
is a basis for $\mathcal{G}\left(p,q\right)$ and assume that in such
basis $Q$ is diagonal, i.e., $Q=\text{diag\ensuremath{\left(\lambda_{1},\lambda_{2},\lambda_{3}\right)}}$
with $\lambda_{1},\lambda_{2},\lambda_{3}\in\left\{ \pm1\right\} $.
We then have the following relations
\begin{align}
e_{i}^{2} & =\lambda_{i},\quad\left(e_{i}e_{j}\right)^{2}=-\lambda_{i}\lambda_{j},\quad\left(e_{1}e_{2}e_{3}\right)^{2}=-\lambda_{1}\lambda_{2}\lambda_{3},\label{eq:SquaredRelations}
\end{align}
where $i\neq j$. Moreover, one also verifies that 

\begin{equation}
\begin{array}{ccc}
\left(e_{1}e_{2}\right)\left(e_{2}e_{3}\right) & = & \lambda_{2}\left(e_{1}e_{3}\right),\\
\left(e_{2}e_{3}\right)\left(e_{1}e_{3}\right) & = & \lambda_{3}\left(e_{1}e_{2}\right),\\
\left(e_{1}e_{3}\right)\left(e_{1}e_{2}\right) & = & \lambda_{1}\left(e_{2}e_{3}\right).
\end{array}\label{eq:QuaternionicRelations}
\end{equation}

Regardless of the basis, in the special case of 3D geometric algebras
there are a few properties of the involutions that will be useful in
the following sections. First, although $xx^{\dagger}\neq x^{\dagger}x$,
nevertheless, since scalars commute with all bivectors in $\mathcal{G}\left(p,q\right)_{+}$
and trivectors commute with all vectors in $\mathcal{G}\left(p,q\right)_{-}$,
we have 
\begin{align}
\left(xx^{\dagger}\right)_{+} & =\left(x^{\dagger}x\right)_{+},\\
\left(xx^{\dagger}\right)_{-} & =\left(x^{\dagger}x\right)_{-},
\end{align}
which will be of paramount importance in the following calculations.
Furthermore, since $xx^{\dagger}$ is equal to its own reverse, i.e.
$xx^{\dagger}=\left(xx^{\dagger}\right)^{\dagger}$, then it must
be either a scalar or a vector, i.e. $xx^{\dagger}=\left\langle xx^{\dagger}\right\rangle _{0,1}$,
so that 
\begin{align}
\left(xx^{\dagger}\right)_{+} & =\left\langle xx^{\dagger}\right\rangle _{0},\\
\left(xx^{\dagger}\right)_{-} & =\left\langle xx^{\dagger}\right\rangle _{1},
\end{align}
which, again will be useful in section \ref{subsec:Octonionic-algebras}.
We are now ready for the recovering of all Hurwitz algebras from 3D
geometric algebras.

\subsection{\label{subsec:Complex-and-quaternionic}Complex and quaternionic
algebras }

Relations (\ref{eq:QuaternionicRelations}) along with $\left(e_{i}e_{j}\right)^{2}=-\lambda_{i}\lambda_{j}$
imply that the \emph{even subalgebra} $\mathcal{G}\left(p,q\right)_+,$
of any 3D geometric algebra $\mathcal{G}\left(p,q\right)$, i.e.,
the subalgebra spanned by the basis $\mathscr{R}=\left\{ 1,e_{1}e_{2},e_{2}e_{3},e_{1}e_{3}\right\} $ endowed with the
geometric product, is a quaternionic or split-quaternionic subalgebra
depending on the signature $\left(p,q\right)$. More specifically,
one obtains that in $\mathcal{G}\left(3,0\right)$ and $\mathcal{G}\left(0,3\right)$ the even subalgebra is isomorphic to that of quaternions, i.e.
$\mathcal{G}\left(p,q\right)_+,\cong\mathbb{H}$. On the other hand,
in $\mathcal{G}\left(1,2\right)$ and $\mathcal{G}\left(2,1\right)$,
the even subalgebra is isomorphic to that of
split-quaternion, $\mathcal{G}\left(p,q\right)_+,\cong\mathbb{H}_{s}$.

Similarly, since 
\begin{equation}
\left(e_{1}e_{2}e_{3}\right)^{2}=-\lambda_{1}\lambda_{2}\lambda_{3},\label{eq:ComplexPseudoscalar}
\end{equation}
the \emph{pseudoscalar subalgebra} $\textit{Ps}\left(\mathcal{G}\right)$ spanned
by the unit and the pseudoscalar, i.e. $\left\{ 1,e_{1}e_{2}e_{3}\right\} $,
is isomorphic to either the complex $\mathbb{C}$ or the split complex
numbers $\mathbb{C}_{s}$ depending on the signature $\left(p,q\right)$.
More specifically, one obtains that in $\mathcal{G}\left(3,0\right)$
and $\mathcal{G}\left(1,2\right)$ the pseudoscalar algebra is isomorphic
to the complex numbers, $\textit{Ps}\left(\mathcal{G}\right)\cong\mathbb{C}$.
On the other hand, in $\mathcal{G}\left(2,1\right)$ and $\mathcal{G}\left(0,3\right)$,
the pseudoscalar algebra is isomorphic to the split-complex numbers, $\textit{Ps}\left(\mathcal{G}\right)\cong\mathbb{C}_{s}$.

It is now important to notice that the trivector $e_{1}e_{2}e_{3}$
commutes with any element of the geometric algebra. Indeed, one has
\begin{align}
\left(e_{1}e_{2}e_{3}\right)e_{i} & =\left(-1\right)^{2}\left(e_{i}\right)\left(e_{1}e_{2}e_{3}\right),\\
\left(e_{1}e_{2}e_{3}\right)\left(e_{i}e_{j}\right) & =\left(-1\right)^{4}\left(e_{i}e_{j}\right)\left(e_{1}e_{2}e_{3}\right).
\end{align}
Therefore, combining all previous arguments one obtains that $\mathcal{G}\left(3,0\right)$
is isomorphic to the algebra of biquaternions $\mathbb{C}\otimes\mathbb{H}$;
the geometric algebra $\mathcal{G}\left(2,1\right)$ is isomorphic
to $\mathbb{C}_{s}\otimes\mathbb{H}_{s}$, $\mathcal{G}\left(1,2\right)$
to $\mathbb{C}\otimes\mathbb{H}_{s}$ and, finally, $\mathcal{G}\left(0,3\right)$
is isomorphic to $\mathbb{C}_{s}\otimes\mathbb{H}$. In all these
cases the isomorphism is concretely realised setting 
\begin{equation}
e_{1}e_{2}=\text{\textbf{i}},e_{2}e_{3}=\text{\textbf{j}},e_{1}e_{3}=\text{\textbf{k}},e_{1}e_{2}e_{3}=\iota.
\end{equation}
 A summary of these relations is found in Table \ref{tab:Correspondence-between-Geometric}.
\begin{table}
\centering{}%
\begin{tabular}{|c|c|c|c|c|c|}
\hline 
Algebra & Tensor product & $\mathcal{G}\left(p,q\right)_+,$ & $\textit{Ps}\left(\mathcal{G}\right)$ & $\left(\mathcal{G}\left(p,q\right),\bullet,N\right)$ & $\left(\mathcal{G}\left(p,q\right),\bullet_{-},N\right)$\tabularnewline
\hline 
\hline 
$\mathcal{G}\left(3,0\right)$ & $\mathbb{C}\otimes\mathbb{H}$ & $\mathbb{H}$ & $\mathbb{C}$ & $\mathbb{O}$ & $\mathbb{O}_{s}$\tabularnewline
\hline 
$\mathcal{G}\left(2,1\right)$ & $\mathbb{C}_{s}\otimes\mathbb{H}_{s}$ & $\mathbb{H}_{s}$ & $\mathbb{C}_{s}$ & $\mathbb{O}_{s}$ & $\mathbb{O}_{s}$\tabularnewline
\hline 
$\mathcal{G}\left(1,2\right)$ & $\mathbb{C}\otimes\mathbb{H}_{s}$ & $\mathbb{H}_{s}$ & $\mathbb{C}$ & $\mathbb{O}_{s}$ & $\mathbb{O}_{s}$\tabularnewline
\hline 
$\mathcal{G}\left(0,3\right)$ & $\mathbb{C}_{s}\otimes\mathbb{H}$ & $\mathbb{H}$ & $\mathbb{C}_{s}$ & $\mathbb{O}_{s}$ & $\mathbb{O}$\tabularnewline
\hline 
\end{tabular}\caption{\label{tab:Correspondence-between-Geometric}Correspondence between
Geometric Algebra, its isomorphic algebra obtained as a tensor product
of two Hurwitz algebras, the even subalgebra $\mathcal{G}\left(p,q\right)_+,$
with basis $\mathscr{R}=\left\{ 1,e_{1}e_{2},e_{2}e_{3},e_{1}e_{3}\right\} $,
and the pseudoscalar subalgebra $\textit{Ps}\left(\mathcal{G}\right)$ with
basis $\left\{ 1,e_{1}e_{2}e_{3}\right\} .$ Finally, on the right
we have included the Hurwitz algebras corresponding to $\mathcal{G}\left(p,q\right)$
with the new products $\bullet$ and $\bullet_{-}$ defined in (\ref{eq:octonionic product})
and (\ref{eq:split-octonionic product-1}) and norm $N$ defined in
(\ref{eq:NormOctonionic}). }
\end{table}

Finally, it is important to point out that a Hurwitz algebra is a
normed algebra, and therefore the definition of such an algebra is
completed only when a norm is defined. Obviously, the above isomorphisms
between geometric subalgebras and Hurwitz algebras implicitly define
a norm that turns these algebras into composition algebras. Nevertheless,
one might ask if there is a way to define such norms without leaving
the geometric algebraic formalism. 

Indeed, we saw in section \ref{subsec:Biquaternionic-algebras} that
every biquaternionic algebra inherits three different conjugations
\begin{equation}
\widetilde{x}=z\otimes\widetilde{q},\,\,\overline{x}=\overline{z}\otimes q,\,\,x^{\dagger}=\overline{z}\otimes\widetilde{q},\label{eq:biquaternionicConj-1}
\end{equation}
where $x\in C\otimes H$ with $C\in\left\{ \mathbb{C},\mathbb{C}_{s}\right\} $,
$H\in\left\{ \mathbb{H},\mathbb{H}_{s}\right\} $ and $z\longrightarrow\overline{z}\in C$,
$q\longrightarrow\widetilde{q}\in H$ indicate the (split-)complex
and the (split-)quaternionic conjugation respectively. 

We now focus on reversion $x\longrightarrow x^{\dagger}$ to see how
it acts on the elements of the basis $\mathscr{B}$. In this case,
we obtain 
\begin{align}
\left(e_{1}e_{2}\right)^{\dagger} & =e_{2}e_{1}=-\text{\textbf{i}},\left(e_{2}e_{3}\right)^{\dagger}=-\text{\textbf{j}},\left(e_{1}e_{3}\right)^{\dagger}=-\text{\textbf{k}},\\
\left(e_{1}e_{2}e_{3}\right)^{\dagger} & =e_{3}e_{2}e_{1}=-\iota,
\end{align}
which corresponds to the biquaternionic conjugation given by $x^{\dagger}=\overline{z}\otimes\widetilde{q}$.
This implies that defining the norm through the reversion, i.e.,
\begin{equation}
n\left(x\right)=xx^{\dagger},
\end{equation}
over $\textit{Ps}\left(\mathcal{G}\right)$ and $\mathcal{G}\left(p,q\right)_+,$
will give the usual complex and quaternionic norm that turns these
algebras into composition algebras. 

Nevertheless, it is interesting to see the action of the other involutions
since the same result can be obtained choosing different conjugations.
The inversion $x\longrightarrow\overline{x}$ acts on the elements
of the basis as
\begin{equation}
\overline{e_{i}}=-e_{i},\quad \overline{e_{i}e_{j}}=e_{i}e_{j},\quad \overline{e_{1}e_{2}e_{3}}=-e_{1}e_{2}e_{3},
\end{equation}
so that in the biquaternionic algebra the inversion corresponds to
the complex conjugation $\overline{x}=\overline{z}\otimes q$. Instead,
Clifford conjugation $x\longrightarrow\widetilde{x}$ acts on the
element of the basis 
\begin{equation}
\widetilde{e_{i}}=-e_{i},\quad \widetilde{e_{i}e_{j}}=e_{j}e_{i}=-e_{i}e_{j},\quad \widetilde{e_{1}e_{2}e_{3}}= -e_{3}e_{2}e_{1} = e_{1}e_{2}e_{3},
\end{equation}
which, in the corresponding biquaternionic algebra, is the quaternionic
conjugation $\widetilde{x}=z\otimes\widetilde{q}$. Thus the three
main involutions of Clifford algebras, i.e., reversion, inversion and
Clifford conjugation, correspond to biquaternionic, complex and quaternionic
conjugation respectively (see Table \ref{tab:Correspondence-between-involutio}).
The fourth involution, i.e. the full grade inversion, will be used
in the construction of the octonionic norm.
\begin{table}
\begin{centering}
\begin{tabular}{|c|c|c|}
\hline 
\textbf{Involution} & \textbf{Geometric algebra} & \textbf{Biquaternionic}\tabularnewline
\hline 
\hline 
$x\longrightarrow x^{\dagger}$ & Reversion & $x^{\dagger}=\overline{z}\otimes\widetilde{q}$\tabularnewline
\hline 
$x\longrightarrow\overline{x}$ & Inversion & $\overline{x}=\overline{z}\otimes q$\tabularnewline
\hline 
$x\longrightarrow\widetilde{x}$ & Clifford conjugation & $\widetilde{x}=z\otimes\widetilde{q}$\tabularnewline
\hline 
\end{tabular}
\caption{\label{tab:Correspondence-between-involutio}Correspondence between
involutions of 3D geometric algebra and different conjugation types
defined over biquaternionic algebras, i.e. $C\otimes H$ with $C\in\left\{ \mathbb{C},\mathbb{C}_{s}\right\} $,
$H\in\left\{ \mathbb{H},\mathbb{H}_{s}\right\} $. In the table we
show how the involutions act on the elements of the algebra $x=z\otimes q\in C\otimes H$
where $z\in C$ and $q\in H$. The fourth involution $x\protect\longrightarrow x^{*}$, called full grade
inversion, corresponds to the octonionic conjugation. }
\par\end{centering}
\end{table}

\subsection{\label{subsec:Octonionic-algebras}Octonionic algebras}

We now proceed recovering the octonionic algebras. The embeddings
that we present here were inspired by Hitzer \cite{Hi22} that extended
to the three-dimensional geometric algebras the work of Lasenby \cite{La22}
over the spacetime algebra (STA) $Cl\left(1,3\right)$. Note that,
since the Clifford algebra product is associative, it does not allow
a direct embedding of the octonionic product, which is non-associative. 

In order to realise the octonionic and the split-octonionic product
with 3D geometric algebras we have to split those algebras into the
disjoint union of two subset
\begin{equation}
\mathcal{G}\left(p,q\right)=\mathcal{G}_{+}\left(p,q\right)\uplus\mathcal{G}_{-}\left(p,q\right),
\end{equation}
where $\mathcal{G}_{+}\left(p,q\right)$ is the set of all multivector
of even grade, while $\mathcal{G}_{-}\left(p,q\right)$ is the set
of all multivector of odd grade. It is worth noting that, while the
set $\mathcal{G}_{+}\left(p,q\right)$ is closed under the Clifford
product and thus forms an algebra that we called $\mathcal{G}\left(p,q\right)_+,$,
the set of $\mathcal{G}_{-}\left(p,q\right)$ does not form an algebra. 

We now proceed with a formalism aimed to enhance the link between
the new construction and the Cayley-Dickson process. In order to do
that, consider an element $x\in\mathcal{G}\left(p,q\right)$ as decomposed
into 
\begin{equation}
x=\left(x_{+},x_{-}\right),
\end{equation}
with $x_{+}\in\mathcal{G}_{+}\left(p,q\right)$ and $x_{-}\in\mathcal{G}_{-}\left(p,q\right)$
so that we can define a new product over $\mathcal{G}\left(p,q\right)$
componentwise. Then, the octonionic product over $\mathcal{G}\left(p,q\right)$
is defined as 
\begin{equation}
\label{eq:octoProductGAOrderedPair}
\left(x_{+},x_{-}\right)\bullet\left(y_{+},y_{-}\right)=\left(x_{+}y_{+}+\widetilde{y_{-}}x_{-},y_{-}x_{+}+x_{-}\widetilde{y_{+}}\right),
\end{equation}
\label{eq:octonionic product}
where $x_{+},y_{+}\in\mathcal{G}_{+}\left(p,q\right)$ and $x_{-},y_{-}\in\mathcal{G}_{-}\left(p,q\right)$.
And the conjugation over the new algebra is given by
\begin{equation}
x\longrightarrow x^{*}=\left(\widetilde{x_{+}},-x_{-}\right).
\end{equation}
The above product is in fact a slight variation of the one introduced
in the Cayley-Dickson process (which is usually applied
to two copies of the same algebra, while here only $\mathcal{G}_{+}\left(p,q\right)$
is a quaternionic algebra) and with the Clifford conjugation $x\longrightarrow\widetilde{x},$
which is the quaternionic conjugation of $\mathcal{G}\left(p,q\right)$. 

The above notation has been introduced to manifest the resemblance
with the Cayley-Dickson process, but for a more geometric algebraic
approach we can define the octonionic product as
\begin{align}
\label{eq:octoProductGAFull}
x\bullet y & =x_{+}y_{+}+\widetilde{y_{-}}x_{-}+y_{-}x_{+}+x_{-}\widetilde{y_{+}}\\
 & =x_{+}y_{+}-y_{-}^{\dagger}x_{-}+y_{-}x_{+}+x_{-}y_{+}^{\dagger}.\nonumber 
\end{align}
Moreover, without any reference to the Cayley-Dickson process one
can consider the full grade inversion $x\longrightarrow x^{*}$ defined
in (\ref{eq:fullGradeInversion}). Then, one defines the norm $N\left(x\right)\in\mathbb{R}$
as 
\begin{equation}
N\left(x\right)=x\bullet x^{*}.\label{eq:NormOcto}
\end{equation}
 To assure the norm is correctly defined, we need the following useful
Lemma
\begin{lem}
\label{lem:wellDefinedNorm}
Let $x\in\mathcal{G}\left(p,q\right)$. The norm \textup{$N\left(x\right)=x\bullet x^{*}$}
is well-defined and is a scalar, i.e., $N\left(x\right)\in\mathbb{R}$.
\end{lem}

\begin{proof}
To show that is well-defined we need to show that $x\bullet x^{*}$
is a scalar. Indeed,
\begin{align}
N\left(x\right) & =x\bullet x^{*}=x\bullet\left(x_{+}^{\dagger}-x_{-}\right)\\
 & =x_{+}x_{+}^{\dagger}-\left(-x_{-}^{\dagger}\right)x_{-}+\left(-x_{-}\right)x_{+}+x_{-}\left(x_{+}^{\dagger}\right)^{\dagger}\nonumber \\
 & =x_{+}x_{+}^{\dagger}+x_{-}^{\dagger}x_{-}-x_{-}x_{+}+x_{-}x_{+}\nonumber \\
 & =(xx^{\dagger})_{+}=\left\langle xx^{\dagger}\right\rangle _{0}.\nonumber 
\end{align}
\end{proof}

\begin{cor}
\label{cor:NormOddEvenDecomp}
Let $x=x_{+}+x_{-},$ with $x\in\mathcal{G}\left(p,q\right)$, $x_{+}\in\mathcal{G}\left(p,q\right)$
and $x_{-}\in\mathcal{G}\left(p,q\right)$. Then 
\begin{equation}
    N\left(x\right) = N(x_{+})+N(x_{-}).
\end{equation}
\end{cor}

\begin{proof}
The expansion of the geometric product in $N\left(x\right)$ immediately
yields 
\begin{align}
\left(xx^{\dagger}\right)_{+} & =\left(\left(x_{+}+x_{-}\right)\left(x_{+}^{\dagger}+x_{-}^{\dagger}\right)\right)_{+}\\
 & =x_{+}x_{+}^{\dagger}+x_{-}x_{-}^{\dagger}=N\left(x_{+}\right)+N\left(x_{-}\right).\nonumber 
\end{align}
\end{proof}
For practical purposes we are interested in an explicit formula for $N\left(x\right)$.
Thus, let $x_{+}=x_{0}+x_{1}e_{1}e_{2}+x_{2}e_{2}e_{3}+x_{3}e_{1}e_{3}$,
$x_{-}=x_{4}e_{1}+x_{5}e_{2}+x_{6}e_{3}+x_{7}e_{1}e_{2}e_{3}$ and
let $\lambda_{1},\lambda_{2},\lambda_{3}\in\left\{ \pm1\right\} $
be the coefficients of the quadratic form for $\mathcal{G}\left(p,q\right)$,
i.e., $Q=\text{diag\ensuremath{\left(\lambda_{1},\lambda_{2},\lambda_{3}\right)}}.$ Then the expansion in Corollary \ref{cor:NormOddEvenDecomp} provides a straightforward way to evaluate the norm. Carrying this out one finds
\begin{equation}
    N\left(x\right)=x_{0}^{2}+x_{1}^{2}\left(\lambda_{1}\lambda_{2}\right)+x_{2}^{2}\left(\lambda_{2}\lambda_{3}\right)+x_{3}^{2}\left(\lambda_{1}\lambda_{3}\right)+x_{4}^{2}\lambda_{1}+x_{5}^{2}\lambda_{2}+x_{6}^{2}\lambda_{3}+x_{7}^{2}\lambda_{1}\lambda_{2}\lambda_{3}.\label{eq:NormOctonionic}
\end{equation}

We are now set for the following theorem.
\begin{thm}
\label{thm:Hurwitz}Let $x,y\in\mathcal{G}\left(p,q\right)$ with
$p+q=3$. Then $\left(\mathcal{G}\left(p,q\right),\bullet,N\right)$
with $\bullet$ and $N$ defined in (\ref{eq:octonionic product})
and (\ref{eq:NormOcto}) is an Hurwitz algebra, and thus isomorphic
to either the octonions or the split-octonions.
\end{thm}

\begin{proof}
First of all a straightforward check shows that the product is bilinear,
the algebra is thus well-defined, is eight-dimensional and the element
$1$ is a unit element for the algebra, i.e., $x\bullet1=1\bullet x=x$
for every $x\in\mathcal{G}\left(p,q\right)$. Therefore $\left(\mathcal{G}\left(p,q\right),\bullet,N\right)$
is an eight-dimensional unital algebra, and to prove the theorem we
just need to show that is a composition algebra, i.e.,
\begin{equation}
N\left(x\bullet y\right)=N\left(x\right)N\left(y\right).
\end{equation}
We proceed with a direct expansion of the norm $N\left(x\right)$.
Indeed, 
\begin{align}
N\left(x\bullet y\right) & =\left\langle \left(x\bullet y\right)\left(x\bullet y\right)^{\dagger}\right\rangle _{0}\\
 & =\left\langle \left(x_{+}y_{+}-y_{-}^{\dagger}x_{-}+y_{-}x_{+}+x_{-}y_{+}^{\dagger}\right)\left(x_{+}y_{+}-y_{-}^{\dagger}x_{-}+y_{-}x_{+}+x_{-}y_{+}^{\dagger}\right)^{\dagger}\right\rangle _{0}.\nonumber 
\end{align}
In the above expansion we are free to discard the odd terms, since
they do not contribute to the grade $0$ part. Thus
\begin{align}
N\left(x\bullet y\right) & =\left\langle x_{+}y_{+}y_{+}^{\dagger}x_{+}^{\dagger}-x_{+}y_{+}x_{-}^{\dagger}y_{-}-y_{-}^{\dagger}x_{-}y_{+}^{\dagger}x_{+}^{\dagger}+y_{-}^{\dagger}x_{-}x_{-}^{\dagger}y_{-}\right.\\
 & \left.+y_{-}x_{+}x_{+}^{\dagger}y_{-}^{\dagger}+y_{-}x_{+}y_{+}x_{-}^{\dagger}+x_{-}y_{+}^{\dagger}x_{+}^{\dagger}y_{-}^{\dagger}+x_{-}y_{+}^{\dagger}y_{+}x_{-}^{\dagger}\right\rangle _{0}\nonumber \\
 & =\left\langle N\left(x_{+}\right)N\left(y_{+}\right)-x_{+}y_{+}x_{-}^{\dagger}y_{-}-y_{-}^{\dagger}x_{-}y_{+}^{\dagger}x_{+}^{\dagger}+N\left(x_{-}\right)N\left(y_{-}\right)\right.\nonumber \\
 & \left.+N\left(x_{+}\right)N\left(y_{-}\right)+y_{-}x_{+}y_{+}x_{-}^{\dagger}+x_{-}y_{+}^{\dagger}x_{+}^{\dagger}y_{-}^{\dagger}+N\left(x_{-}\right)N\left(y_{+}\right)\right\rangle _{0}.\nonumber 
\end{align}
Now, since $\left\langle xy\right\rangle _{0}=\left\langle yx\right\rangle _{0}$
for any multivectors $x,y$, the non-norm terms in this expression cancel
each other out leaving
\begin{align}
N\left(x\bullet y\right) & =\left(N\left(x_{+}\right)+N\left(x_{-}\right)\right)\left(N\left(y_{+}\right)+N\left(y_{-}\right)\right)\\
 & =N\left(x\right)N\left(y\right).\nonumber 
\end{align}
Therefore, $\left(\mathcal{G}\left(p,q\right),\bullet,N\right)$ being a unital
composition algebra of dimension 8, the only two possibilities given by 
the Hurwitz theorem are that $\left(\mathcal{G}\left(p,q\right),\bullet,N\right)$
is isomorphic to either the octonions $\mathbb{O}$ or the split octonions
$\mathbb{O}_{s}$. 
\end{proof}
For completeness, we also produce an explicit formula for the $N\left(x\bullet y\right)$.
Setting 
\begin{equation}
\begin{array}{cc}
x_{+} & =x_{0}+x_{1}e_{1}e_{2}+x_{2}e_{2}e_{3}+x_{3}e_{1}e_{3},\\
x_{-} & =x_{4}e_{1}+x_{5}e_{2}+x_{6}e_{3}+x_{7}e_{1}e_{2}e_{3},\\
y_{+} & =y_{0}+y_{1}e_{1}e_{2}+y_{2}e_{2}e_{3}+y_{3}e_{1}e_{3},\\
y_{-} & =y_{4}e_{1}+y_{5}e_{2}+y_{6}e_{3}+y_{7}e_{1}e_{2}e_{3},
\end{array}
\end{equation}
after a long and tedious calculation and considering that $\lambda_{i}^{2}=1$
for $i\in\left\{ 1,2,3\right\} $, one obtains 
\begin{align}
N\left(x\bullet y\right) & =\left(x_{0}^{2}y_{0}^{2}+x_{1}^{2}y_{1}^{2}+x_{2}^{2}y_{2}^{2}+x_{3}^{2}y_{3}^{2}+x_{4}^{2}y_{4}^{2}+x_{5}^{2}y_{5}^{2}+x_{6}^{2}y_{6}^{2}+x_{7}^{2}y_{7}^{2}\right)+\\
 & +\left(x_{4}^{2}y_{0}^{2}+x_{5}^{2}y_{1}^{2}+x_{6}^{2}y_{3}^{2}+x_{0}^{2}y_{4}^{2}+x_{1}^{2}y_{5}^{2}+x_{3}^{2}y_{6}^{2}+x_{7}^{2}y_{2}^{2}+x_{2}^{2}y_{7}^{2}\right)\lambda_{1}\nonumber \\
 & +\left(x_{5}^{2}y_{0}^{2}+x_{6}^{2}y_{2}^{2}+x_{4}^{2}y_{1}^{2}+x_{1}^{2}y_{4}^{2}+x_{7}^{2}y_{3}^{2}+x_{0}^{2}y_{5}^{2}+x_{2}^{2}y_{6}^{2}+x_{3}^{2}y_{7}^{2}\right)\lambda_{2}\nonumber \\
 & +\left(x_{6}^{2}y_{0}^{2}+x_{5}^{2}y_{2}^{2}+x_{2}^{2}y_{5}^{2}+x_{3}^{2}y_{4}^{2}+x_{0}^{2}y_{6}^{2}+x_{4}^{2}y_{3}^{2}+x_{7}^{2}y_{1}^{2}+x_{1}^{2}y_{7}^{2}\right)\lambda_{3}\nonumber \\
 & +\left(x_{0}^{2}y_{1}^{2}+x_{1}^{2}y_{0}^{2}+x_{2}^{2}y_{3}^{2}+x_{5}^{2}y_{4}^{2}+x_{4}^{2}y_{5}^{2}+x_{7}^{2}y_{6}^{2}+x_{6}^{2}y_{7}^{2}+x_{3}^{2}y_{2}^{2}\right)\lambda_{1}\lambda_{2}\nonumber \\
 & +\left(x_{3}^{2}y_{1}^{2}+x_{2}^{2}y_{0}^{2}+x_{1}^{2}y_{3}^{2}+x_{6}^{2}y_{5}^{2}+x_{5}^{2}y_{6}^{2}+x_{4}^{2}y_{7}^{2}+x_{0}^{2}y_{2}^{2}+x_{0}^{2}y_{2}^{2}\right)\lambda_{2}\lambda_{3}\nonumber \\
 & +\left(x_{3}^{2}y_{0}^{2}+x_{2}^{2}y_{1}^{2}+x_{1}^{2}y_{2}^{2}+x_{0}^{2}y_{3}^{2}+x_{7}^{2}y_{5}^{2}+x_{4}^{2}y_{6}^{2}+x_{5}^{2}y_{7}^{2}+x_{6}^{2}y_{4}^{2}\right)\lambda_{1}\lambda_{3}\nonumber \\
 & +\left(x_{7}^{2}y_{0}^{2}+x_{6}^{2}y_{1}^{2}+x_{4}^{2}y_{2}^{2}+x_{5}^{2}y_{3}^{2}+x_{2}^{2}y_{4}^{2}+x_{3}^{2}y_{5}^{2}+x_{1}^{2}y_{6}^{2}+x_{0}^{2}y_{7}^{2}\right)\lambda_{1}\lambda_{2}\lambda_{3}\nonumber \\
 & =N\left(x\right)N\left(y\right).\nonumber 
\end{align}

Theorem \ref{thm:Hurwitz} proved that the geometric algebra endowed
with the bilinear product $\bullet$ and the norm $N$ obtained from
the full grade inversion, i.e. the triple $\left(\mathcal{G}\left(p,q\right),\bullet,N\right)$,
is isomorphic to either the octonions or the splitoctonions. To discriminate
between them it is sufficient to look at the full expression of the
norm, i.e., 
\begin{equation}
N\left(x\right)=x_{0}^{2}+x_{1}^{2}\left(\lambda_{1}\lambda_{2}\right)+x_{2}^{2}\left(\lambda_{2}\lambda_{3}\right)+x_{3}^{2}\left(\lambda_{1}\lambda_{3}\right)+x_{4}^{2}\lambda_{1}+x_{5}^{2}\lambda_{2}+x_{6}^{2}\lambda_{3}+x_{7}^{2}\lambda_{1}\lambda_{2}\lambda_{3}.
\end{equation}
Indeed, in a division algebra such as the octonions $\mathbb{O}$
an element $x$ is of zero-norm $N\left(x\right)=0$ if and only if $x=0$.
This is the case of $\mathcal{G}\left(3,0\right)$, for which $\left(\lambda_{1},\lambda_{2},\lambda_{3}\right)=\left(1,1,1\right)$
and
\begin{equation}
N\left(x\right)=x_{0}^{2}+x_{1}^{2}+x_{2}^{2}+x_{3}^{2}+x_{4}^{2}+x_{5}^{2}+x_{6}^{2}+x_{7}^{2},
\end{equation}
which is the norm of octonions $\mathbb{O}$. Setting $\left(\lambda_{1},\lambda_{2},\lambda_{3}\right)=\left(-1,-1,-1\right)$,
which is the case for $\mathcal{G}\left(0,3\right)$, we do have 
\begin{equation}
N\left(x\right)=x_{0}^{2}+x_{1}^{2}+x_{2}^{2}+x_{3}^{2}-x_{4}^{2}-x_{5}^{2}-x_{6}^{2}-x_{7}^{2},
\end{equation}
which is the norm of split-octonions $\mathbb{O}_{s}$. In the case
of $\mathcal{G}\left(2,1\right)$, for which $\left(\lambda_{1},\lambda_{2},\lambda_{3}\right)=\left(1,1,-1\right)$,
we have 

\begin{equation}
N\left(x\right)=x_{0}^{2}+x_{1}^{2}-x_{2}^{2}-x_{3}^{2}+x_{4}^{2}+x_{5}^{2}-x_{6}^{2}-x_{7}^{2},
\end{equation}
as happens also in the case of $\mathcal{G}\left(1,2\right)$, for which
$\left(\lambda_{1},\lambda_{2},\lambda_{3}\right)=\left(-1,-1,1\right)$,
giving
\begin{equation}
N\left(x\right)=x_{0}^{2}+x_{1}^{2}-x_{2}^{2}-x_{3}^{2}-x_{4}^{2}-x_{5}^{2}\lambda_{2}+x_{6}^{2}+x_{7}^{2}.
\end{equation}

Finally, it is worth noting that the same process and line of reasoning
could be done defining, over the three-dimensional
geometric algebras, another product which differs from $\bullet$ in just one sign, specifically 
\begin{equation}
x\bullet_{-}y=\left(x_{+},x_{-}\right)\bullet_{-}\left(y_{+},y_{-}\right)=\left(x_{+}y_{+}-\widetilde{y_{-}}x_{-},y_{-}x_{+}+x_{-}\widetilde{y_{+}}\right),\label{eq:split-octonionic product-1}
\end{equation}
where $x_{+},y_{+}\in\mathcal{G}_{+}\left(p,q\right)$ and $x_{-},y_{-}\in\mathcal{G}_{-}\left(p,q\right)$.
In fact, a glance at (\ref{eq:octoProductGAFull}) indicates that this amounts to defining the product with the same signs but using reversion instead of Clifford conjugation as the involution. In this case, following the same steps with a small variation, we
obtain the new full form of the norm
\begin{equation}
N\left(x\right)=x_{0}^{2}+x_{1}^{2}\left(\lambda_{1}\lambda_{2}\right)+x_{2}^{2}\left(\lambda_{2}\lambda_{3}\right)+x_{3}^{2}\left(\lambda_{1}\lambda_{3}\right)-x_{4}^{2}\lambda_{1}-x_{5}^{2}\lambda_{2}-x_{6}^{2}\lambda_{3}-x_{7}^{2}\lambda_{1}\lambda_{2}\lambda_{3}.
\end{equation}

Again one obtains that $\left(\mathcal{G}\left(p,q\right),\bullet_{-},N\right)$
is a composition algebra, but this time we have that in case $\mathcal{G}\left(3,0\right)$
and $\left(\lambda_{1},\lambda_{2},\lambda_{3}\right)=\left(1,1,1\right)$
the resulting algebra is isomorphic to that of the split-octonions
$\mathbb{O}_{s}$, as in the case of $\mathcal{G}\left(2,1\right)$ and
$\mathcal{G}\left(1,2\right)$ with $\left(\lambda_{1},\lambda_{2},\lambda_{3}\right)=\left(1,1,-1\right)$
and $\left(\lambda_{1},\lambda_{2},\lambda_{3}\right)=\left(-1,-1,1\right)$
respectively. However, for $\mathcal{G}\left(0,3\right)$ and
$\left(\lambda_{1},\lambda_{2},\lambda_{3}\right)=\left(-1,-1,-1\right)$
the resulting algebra is that of the octonions $\mathbb{O}$. A summary
of the correspondence can be found on the left of Table \ref{tab:Correspondence-between-Geometric}. 

\section{\label{sec:Conclusions-and-further}Conclusions and further developments}

In this paper we show how all seven Hurwitz algebras, i.e. $\mathbb{R},\mathbb{C},\mathbb{H},\mathbb{O}$
and their split companions, can be realized within 3D geometric algebras
$\mathcal{G}\left(p,q\right)$ with $p+q=3$. The result is achieved
considering the even subalgebra $\mathcal{G}\left(p,q\right)_+,$,
which is always of quaternionic or split-quaternionic type, and the
algebra generated by the pseudoscalar $\textit{Ps}\left(\mathcal{G}\right)$
which is always of complex or split-complex type. Since the pseudoscalar
algebra $\textit{Ps}\left(\mathcal{G}\right)$ commutes with all elements of
$\mathcal{G}\left(p,q\right)$, then one obtains the isomorphism between
the 3D geometric algebras and the biquaternionic algebras as shown
in Table \ref{tab:Correspondence-between-Geometric}. Then the main
involutions of 3D geometric algebras, i.e., reversion, inversion and
Clifford conjugation, correspond to biquaternionic, complex and quaternionic
conjugation respectively (see Table \ref{tab:Correspondence-between-involutio}).

The embedding of the octonionic algebras, instead, is achieved defining
a new product over the 8-dimensional algebras through a variation
of the Cayley-Dickson process. Octonionic conjugation is then obtained
using the fourth natural involution over geometric algebras, i.e.
the full grade inversion. This approach is quite insightful and we
are confident that it will lead to geometric algebra realisation not
only of Hurwitz algebras but also of para-Hurwitz algebras and of
all symmetric composition algebras.

\section{Acknowledgment}

Authors would like to thank David Chester and Alessio Marrani for
useful converstions and suggestions.

$*$\noun{ Departamento de Matematica, }\\
\noun{Universidade do Algarve, }\\
\noun{Campus de Gambelas, }\\
\noun{8005-139 Faro, Portugal} 
\begin{verbatim}
a55499@ualg.pt

\end{verbatim}
$**$\noun{Quantum Gravity Research,}

\noun{Topanga Canyon Rd 101 S., }

\noun{California CA, }

\noun{90290, USA}
\begin{verbatim}
Richard@QuantumGravityResearch.org
Klee@QuantumGravityResearch.org

\end{verbatim}

\end{document}